\documentclass{amsart}
\usepackage{amstext, amssymb, amsthm, amsfonts, latexsym,bbm,tikz}
\usetikzlibrary{matrix}
\usepackage{varioref}
\usepackage{mathtools}

\usepackage{enumitem}
\usepackage{todonotes}
\usepackage{bbm}
\usepackage[utf8]{inputenc} 
\usepackage[T1]{fontenc}

\usepackage{comment}
\usepackage{tikz-cd}
\usepackage{hyperref}

\newtheorem{theorem}{Theorem}[section]
\newtheorem{lemma}[theorem]{Lemma}
\newtheorem{corollary}[theorem]{Corollary}
\newtheorem{proposition}[theorem]{Proposition}

\theoremstyle{definition}
\newtheorem{definition}[theorem]{Definition}
\newtheorem{example}[theorem]{Example}

\theoremstyle{remark}
\newtheorem{remark}[theorem]{Remark}

\theoremstyle{Conjecture/open problem}

\def\a{\alpha}

\def\la{\lambda}

\def\a{\alpha}
\def\be{\beta}

\def\Ga{\Gamma}
\def\1{\mathbf{1}}
\def\la{\lambda}

\def\p{^{\prime}}

\def\k{\kappa}

\def\cG{\mathcal{G}}
\def\cR{\mathcal{R}}
\def\cS{\mathcal{S}}
\def\cC{\mathcal{C}}

\def\R{\mathbb{R}}

\def\Z{\mathbb{Z}}

\DeclareMathOperator*{\too}{\to}
\DeclareMathOperator*{\rightleftharpoon}{\rightleftharpoons}

\numberwithin{equation}{section}

\begin{document}

\title{Stationary distributions via decomposition of stochastic reaction networks}
\hfill\break

\author[Linard Hoessly]{Linard Hoessly}
\address{Department of Mathematical Sciences, University of Copenhagen, Denmark}
\email{hoessly@math.ku.dk}
\date{}

\begin{abstract}
We examine reaction networks (CRNs) through their associated continuous-time Markov processes. Studying the dynamics of such networks is in general hard, both analytically and by simulation. In particular, stationary distributions of stochastic reaction networks are only known in some cases. We analyze class properties of the underlying continuous-time Markov chain of CRNs under the operation of join and examine conditions such that the form of the stationary distributions of a CRN is derived from the parts of the decomposed CRNs. The conditions can be easily checked in examples and allow recursive application. The theory developed enables sequential decomposition of the Markov processes and calculations of stationary distributions. Since the class of processes expressible through such networks is big and only few assumptions are made, the principle also applies to other stochastic models. 
We give examples of interest from CRN theory to highlight the decomposition.
\end{abstract}

\maketitle

\section{Introduction}
Reaction networks (CRNs) form a broadly applicable paradigm to describe the interactions of different constituents through mathematical models.
 CRNs are vital 
 for the prediction and analysis of data in biochemistry, systems biology and cellular biology, 
and have found further applications \cite{ecol_May,sociology,Goutsias_ov}. 
Besides their relevance in applications, 
CRNs continue to drive the development of areas of mathematics such as
dynamical systems theory, stochastic processes and applied algebraic geometry \cite{anderson2,overv_Gorban}.

A CRN consists of reactions with associated reaction rates that govern the speed of the reactions. CRNs are often defined via the reaction graph, that highlights the interactions between species and their transformations.
As an example consider the enzymatic Michaelis-Menten mechanism, where an enzyme $E$ catalyzes the conversion of a substrate $S$ into a product $P$ through an intermediate molecule $ES$: 
\begin{equation}\label{react_1}S+E\rightleftharpoons ES\to P+E.\end{equation}
Either a deterministic or a stochastic model is chosen 
 to represent the dynamics of CRNs. Traditionally, deterministic models have been the preferred modelling choice. However, with the emergence of systems biology, cellular biology and synthetic biology the importance of modelling systems with small molecular counts have become important.
Stochastic models of CRNs are used when the molecular counts in the system are low. They typically consist of continuous time Markov chains (CTMC), which 
apply to many processes in living systems \cite{Gardiner,Goutsias_ov,Mely_Khammash}.
Furthermore the efficient mathematical
analysis of their stochastic properties is an invaluable tool for their application. Two realms of investigation are generally of interest for such 
 systems. The \emph{transient behaviour} describes the time-dependent dynamics, whereas the \emph{stationary behaviour} describes  the dynamics in the long term after the system has reached an equilibrium.
\hfill\break

Studying the dynamics of stochastic CRNs is difficult in general, and so they are often examined via simulations \cite{gillespieb}. 
The stationary behaviour and its characterization are typically analysed via the master equation. In many cases, the stationary behavior of Markov chains can be described through their stationary distribution. Exact solutions for the stationary distribution (if it exists) are not known for most systems, except for some special cases. Complex balanced reaction networks are fairly well understood by now.
Deterministic complex balanced CRN have their stochastic counterparts with product-form stationary distributions of Poisson-type \cite{anderson2}.
The reverse statement is essentially also true: a stochastic CRN with product-form stationary distribution of Poisson-type (on any irreducible component)  is complex balanced \cite{Cappelletti}. Complex balanced CRNs are in particular weakly-reversible. Apart from that, there are some results on form of stationary distributions of non-weakly reversible reaction networks, like, e.g. autocatalytic CRN \cite{hoesslysta}. 

Here, we study unions (or, joins) of reaction networks in the stochastic setting. Our main focus is the form and existence of stationary distributions. While \cite{hoesslysta} focussed on a particular class of interest of non-weakly reversible CRNs with applications in particle systems, life sciences and condensation, we generalise here the underlying proof principle for stationary distributions. We give tools to systematically find the stationary distributions for the joined CRN, given the stationary distributions of the smaller CRNs. To be more precise, in CRNs where the stationary distributions of the decompositions are of product form and concur in the species in common, we can derive the stationary distribution of the full CRN from its parts. These are sufficient conditions, and examples can come from any combination of CRNs as long as the stationary distributions are of product form and satisfy some condition on the state spaces. Since the class expressible through such networks is big (i.e. interacting particle systems, cf., e.g., \cite{ligget}), the principle also applies to other stochastic models. As an example, consider \cite{hoesslysta} for the relation to the inclusion process. In particular, autocatalytic CRNs and more general non-weakly reversible as well as some weakly reversible (including all complex balanced) CRNs fall under the framework we consider. 

\color{black}
One result is then that given a reaction network $\cG$ that can be decomposed as a reaction-disjoint union $\cG=\cG_1 \cup \cG_2$, with $\cG_1,\cG_2$ essentials and of product form stationary distributions such that the product-form functions agree in the species in common, the stationary distribution of $\cG$ is of product form and derived from $\cG_1,\cG_2$ under a summability condition.
\color{black}

As an illustration consider the following CRN with Mass-action kinetics.

$$
 S_1 \rightleftharpoon^{\k_1}_{\k_2} S_2 \rightleftharpoon^{\k_5}_{\k_6} S_3,\quad 2S_1 \rightleftharpoon^{\k_3}_{\k_4} S_1+S_2,\quad 2S_3 \rightleftharpoon^{\k_7}_{\k_8} S_2+S_3
 $$
 Then, taking as $\cG_1$ the reactions between $S_1,S_2,$ and $\cG_2$ the reactions between $S_2,S_3$, we can apply the above to derive the product-form stationary distribution of $\cG=\cG_1 \cup \cG_2$ for all positive rate constants. The stationary distribution is (see Example \ref{ex_class2})
 $$\pi(x_1,x_2,x_3)=\frac{1}{Z}f_1(x_1)f_2(x_2)f_3(x_3),$$
 where the product form functions are
 $$f_1(x_1)=\frac{1}{x_1!}\prod_{l=1}^{x_1} \frac{\k_2+\k_4(l-1)}{\k_1+\k_3(l-1)} ,\quad f_2(x_2)=\frac{1}{x_2!},\quad f_3(x_3)=\frac{1}{x_3!}\prod_{l=1}^{x_3} \frac{\k_6+\k_8(l-1)}{\k_5+\k_7(l-1)}.$$
 As the overall CRN $\cG=\cG_1 \cup \cG_2$ is reversible and of deficiency two, such examples show that weakly reversible non-complex balanced CRNs can have product-form stationary distributions.

As another example consider the next CRN that can be decomposed in a complex balanced (reactions between $S_3,S_5$ and between $S_1,S_3$) and a join of a non-weakly reversible and a weakly reversible non-complex balanced CRN(the rest)
 $$S_1\rightleftharpoons S_2 \rightleftharpoon S_3 \rightleftharpoons S_4, $$
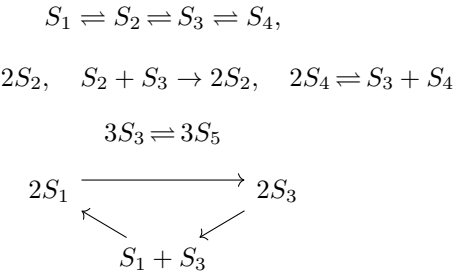
$$S_1+S_2\to 2S_2,\quad S_2+S_3\to 2S_2, \quad 2S_4 \rightleftharpoon S_3+S_4$$
$$3S_3\rightleftharpoon 3S_5$$
$$\begin{tikzcd} [ row sep=1em,
  column sep=1em]
 2S_1 \arrow[rr, shift left=1ex,] 
&& 2S_3 \arrow[dl, ] \\
& S_1+S_3\arrow[ul,]
\end{tikzcd}
$$ with product-form stationary distributions (see Example \ref{more_compl_cx})
$$\pi(x_1,x_2,x_3,x_4,x_5)=\frac{1}{Z}f_1(x_1)f_2(x_2)f_3(x_3)f_4(x_4)f_5(x_5)$$
with $f_1,f_3,f_5$ of Poisson-form, $f_2$ of a form from autocatalytic CRNs and $f_4$ as $f_1,f_2$ of the previous example.

\subsection*{Structure}
In Section \ref{CRN} we introduce basic definitions and terminology for reaction networks. Then we introduce the models for CRNs in Section \ref{modelCRN} and focus on the stochastic model by reviewing definitions, properties, and results on stationary distributions, where at the end we introduce unions of CRNs. In Section
\ref{stochCRN} we study stochastic CRNs under joins and give some results on extending the stationary distributions from smaller CRNs to their joins. Section \ref{stochCRNEx} introduces some examples to outline the application of the developed theory.
\subsection*{Relation to existing approaches}
Previous approaches for extending analytical results on stationary distributions for reaction networks have focussed on gluing one state \cite{Mely_Khammash} or two states \cite{Mely_Pfa} of finite irreducible CTMCs.
\subsection*{Acknowledgements}
The author was supported by the Swiss National Science Foundation Early Postdoc.Mobility grant P2FRP2\_188023, and acknowledges the friendly hospitality of the department of Mathematics at the University of Copenhagen. We thank Christian Mazza, Carsten Wiuf and Chuang Xu for helpful discussions, \color{black}as well as three anonymous referees for valuable feedback and suggestions.
\color{black}

\section{Reaction networks}\label{CRN}
A \emph{reaction network} $\cG$ consists of a finite set of \emph{species} $\cS=\{S_1,\cdots,S_n\}$, a finite set of \emph{complexes}, and a finite set of \emph{reactions} $\cR$, which is then denoted as the triple $\cG=(\cS,\cC,\cR)$.

 We represent the complexes by vectors in $\Z^n_{\geq 0}$, and write reactions as $\nu\to \nu\p $, where we assume $\nu,\nu\p\in\cC$ and $\nu\neq\nu\p$ for all $\nu\to \nu\p \in\cR$.
For a reaction $\nu\to \nu\p $, $\nu$ is called the \emph{reactant} and $\nu\p$ the \emph{product}. 
Every reaction $\nu\to \nu\p$ has a positive \emph{rate constant} $\k_{\nu\to\nu\p}$ associated. Then, given the vector of reaction rates $\k\in\R_{>0}^\cR$, we denote
the CRN with rates by $(\cG,\k)$.
\subsection{Basic terminology}

We illustrate reaction networks by their \emph{reaction graph}, which is the directed graph obtained by taking the vertices $\cC$ and arrows $\cR$. Connected components of the reaction graph are called \emph{linkage classes}.  A CRN is \emph{reversible} if for every $\nu\to \nu\p \in\cR$, $ \nu\p\to \nu \in\cR$. A CRN is \emph{weakly reversible} if for any reaction $\nu\to \nu\p \in\cR$, there is a directed path in the reaction graph that begins with $\nu\p$ and ends in $\nu$. The \emph{molecularity} of a reaction $\nu\to \nu\p \in\cR$ is equal to the number of molecules in the reactant $|\nu|=\sum_i \nu_i$, and correspondingly we say such reactions are unimolecular, bimolecular, three-molecular or $n$-molecular. The \emph{stochiometric subspace} spans a subspace of $\mathbb{R}^n$ and is given as $\mathcal{T}={\rm span}_{\nu\to\nu\p\in\cR}\{\nu-\nu\p\}\subset \mathbb{R}^n$.
The \emph{deficiency}  of a reaction network $\cG$ is given by
 $\delta=|\cC|-\ell-{\rm dim}(\mathcal{T}),$
 where $\ell$ is the number of linkage classes.
 A CRN  $\cG$  is \emph{conservative} if there is a vector $c\in \R^\cS_{>0}$ such that for any reaction $\nu\to\nu\p\in \cR$ we have $\sum_{i\in\cS}\nu_i c_i=\sum_{i\in\cS}\nu\p_i c_i$.
  \color{black}

  \section{Models and kinetics for reaction networks}\label{modelCRN}


  \subsection{Stochastic model}\label{stoch_model}
 The progression of species counts is described  
by a vector $X(t)=x\in \Z_{\geq 0}^n$, which changes according to the 'firing' of the reactions $\nu\to \nu\p$  by jumping from $x$ to $x+\nu\p-\nu$  with transition intensity $\la_{\nu \to \nu\p}(x)$.
The Markov process with intensity functions $\la_{\nu \to \nu\p}:\Z_{\geq 0}^n\to \R_{\geq 0}$ can then be given by
$$P(X(t+\Delta t)=x+\nu\p-\nu|X(t)=x)=\sum_{\nu\to\nu\p\in \cR|-\nu + \nu\p=\xi}\la_{\nu \to \nu\p}(x)\Delta t+ o(\Delta t),$$
with the generator $\mathcal{A}$ acting by
$$\mathcal{A}h(x)=\sum_{\nu \to \nu\p\in \cR} \la_{\nu \to \nu\p}(x)(h(x+\nu\p-\nu)-h(x)),$$
for $h:\Z^n\to \R$.

The transition intensity under mass-action kinetics (more general kinetics are possible as well \cite{anderson2,non-stand_2}) associated to the reaction $\nu\to\nu\p$  is 
\begin{equation}
\label{int}\la_{\nu \to \nu\p}(x)=\k_{\nu\to\nu\p}\frac{(x)!}{(x-\nu)!}\1_{x\geq \nu},
\end{equation}
\color{black}where $z!:=\prod_{i=1}^nz_i!\text{ for } z\in \Z^n_{\geq 0},$ and $x\geq \nu$ if and only if this holds for every component, i.e. $x_i\geq\nu_i$ $\forall S_i\in \cS$.

 \color{black}General inquiry into stochastic CRNs proceeds by inspection of the underlying CTMC. After identifying the class structure and the (so-called) stoichiometric compatibility classes where the dynamics is confined to, the state space is decomposed into different types of states ( cf., i.e., \cite{norris}). On irreducible components, positive recurrence is equivalent to non-explositivity together with existence of an invariant distribution \cite{norris}.

The classification and description of the stochastic behaviour of CRNs is complex. Many interesting results were investigated, like positive recurrence \cite{andersonpos,Tier},  non-explositivity of complex balanced CRN \cite{non-expl}, extinction/absorption events \cite{extinction,quasi_stat}, quasi-stationary distributions \cite{quasi_stat} or the classification of states of stochastic CRNs \cite{xu2019dynamics}. However, even in situations where theorems apply, we are far from a complete characterization, see \cite{andersonpos,non-expl,extinction,quasi_stat,xu2019dynamics} for examples.

 We next introduce some terminology for stochastic reaction networks. A reaction $y\to y\p$ is \emph{active} on $x\in \Z_{\geq 0}^n$ if $x\geq y$. Similarly a reaction $y\to y\p$ is \emph{active} on a set $A\subseteq  \Z_{\geq 0}^n$ if there is a state $x\in A$ such that the reaction is active on $x$. This will mostly be used for $A=\Ga$ an irreducible component. A state $u\in \Z_{\geq 0}^n$ is \emph{accessible} from $x\in \Z_{\geq 0}^n$ if it can be reached from $x$ via the underlying CTMC.
 We will denote this by $x\to_{\cG}u$.

 A non-empty set $\Ga\subset\Z_{\geq 0}^n$ is an \emph{irreducible component} of $\cG$ if for all $x\in \Ga$ and all $u\in  \Z_{\geq 0}^n$, $u$ is accessible from $x$ if and only if $u\in \Ga$.

 We say $\cG$ is \emph{essential} if the state space is a union of irreducible components, and
 $\cG$ is \emph{almost essential} if the state space is a union of irreducible components except for a finite number of states.
\subsection{Stationary distributions of reaction networks} \label{stat}
Let $X(t)$ denote the underlying stochastic process associated to a reaction network on an irreducible component $\Ga$. Then, given that the stochastic process $X(t)$ is positive recurrent and starts in $\Ga$, we have that the limiting distribution is the stationary distribution, i.e.
$$\lim_{t\to\infty}P(X(t)\in A)=\pi_\Ga(A),\text{for any } A\subset\Ga.$$
In particular, if the underlying CTMC is positive recurrent, the stationary distribution $\pi_\Ga$ on an irreducible component $\Ga$ is unique and describes the long-term behavior ( cf., e.g. \cite{norris}). 

The stationary distribution is determined by the master equation of the underlying Markov chain:
\begin{equation} \label{master_eq1}
  \sum_{\nu\to\nu\p\in\cR} \pi(x+\nu-\nu\p)\la_{\nu \to \nu\p}(x+\nu -\nu\p)=\pi(x)\sum_{\nu\to\nu\p\in\cR}\la_{\nu \to \nu\p}(x),
  \end{equation}
for all $x\in\Ga$. A popular choice as rate function is mass-action kinetics, which then gives the following master equation:

\begin{equation} \label{master_eq2}  \sum_{\nu\to\nu\p\in\cR} \pi(x+\nu-\nu\p)\k_{\nu\to\nu\p}\frac{(x-\nu\p+\nu)!}{(x-\nu\p)!}\1_{x\geq \nu\p}=\pi(x)\sum_{\nu\to\nu\p\in\cR}\k_{\nu\to\nu\p}\frac{(x)!}{(x-\nu)!}\1_{x\geq \nu}.\end{equation}

Solving equation (\ref{master_eq1}) is in general a challenging task, even when restricting to the mass-action case. 
\begin{remark}\label{rmk_pos_cons}
Observe that for conservative CRNs the irreducible components are finite. Therefore the CTMC dynamics are positive recurrent (e.g., by Reuters criterion, c.f., e.g. \cite{norris}) on these irreducible components and the limiting distribution is the unique stationary distribution.{ \color{black}Recall in particular that for infinite CTMCs existence of stationary distribution does not imply positive recurrence, cf., e.g. \cite[Ex 3.5.4]{norris} or \cite{non-expl}. 
}
\end{remark}

\subsection{Known results on stationary distributions} \label{res_stat}
 Studying transient and stationary behaviour of reaction networks
are formidable tasks in general, and they are often examined via simulations \cite{gillespieb}. 
Analytical solutions for the stationary distribution (if it exists) are not known for most systems, except for some special cases. 

Some stationary distributions of weakly reversible reaction networks are well-understood. Complex balanced CRNs have a Poisson product-form stationary distribution \cite{anderson2} and can even be characterized by that.
For $(\cG,\k)$ a complex balanced CRN and an irreducible component $\Ga$, the stochastic system has product-form stationary distribution of the form
 $$\pi(x)=M_\Ga \frac{c^x}{x!},x\in\Ga,$$
 where $c\in\R^n_{>0}$ is a point of complex balance, $c^x:=\prod_{S_i\in\cS}c_i^{x_i}$, and $M_\Ga$ is a normalizing constant.
On the other hand, by \cite[Theorem 5.1]{Cappelletti} any almost essential stochastic reaction network with product-form stationary distribution of Poisson-type (i.e. in the form as above) is deterministically complex balanced. Notice that since complex balanced implies weakly reversible, these results do not apply to non-weakly reversible CRNs. Results on both product-form stationary distribution and connection to the deterministic system extend to non-mass action kinetics \cite{anderson2,non-stand_2}. 
Hence complex balanced CRNs are fairly well-understood.

For other classes of CRNs some results are also known \cite{hoesslysta}, i.e. so-called autocatalytic CRNs, a class of non-weakly-reversible CRNs also have product-form stationary distributions. Their product form functions come from an infinite family of functions, where the first one specializes to the Poisson form as above. So for
a autocatalytic CRN in the sense of \cite[$\S$ 3]{hoesslysta},
the stochastic dynamics
has the product-form stationary distribution
\begin{equation}\label{MainFormula}
\pi(x)=Z_\Ga^{-1}\prod_{S_i\in\cS}^{ }f_i(x_i) ,
\end{equation}
with product-form functions
$$f_i(x_i)=\lambda_i^{x_i}\frac{1}{x_i!}\prod_{l=1}^{x_i}(1+\sum_{k=2}^{n_i}\be_i^k \prod_{r=1}^{k-1}(l-r))$$ 
on its irreducible components ($\lambda_i$ and $\be_i^k$ are determined by the autocatalytic CRN, cf. \cite[$\S$ 3]{hoesslysta}) and with $Z_\Ga$ the normalising constant. Some other results on the stochastic behavior of CRN beyond complex balance are in \cite{bibb_stat} or \cite{irrev}.

Beyond these results little is known concerning explicit stationary distributions.

\subsection{Balance equations for stationary distributions of CRNs}\hfill\break
We start with a general definition for balance equations, and recover some classical notions in Remark \ref{rmk_gen}.
\color{black}
The definition below essentially states that stationary distributions factorise according to a decomposition of the reactions of the underlying CRN \cite{hoesslysta}. 
\color{black}

\begin{definition}\label{tauto_def} Consider a CRN $(\cG,\k)$ with stochastic dynamics on $\Ga$ and $\pi$ a stationary distribution on $\Ga$.
We say $(\cG,\k)$ is generalized balanced for $\pi$ on $\Ga$ if there exists $\{(L_i,R_i)_{i\in A}\}$ a set of tuples of subsets of $\cR$ with $A$ an index set such that
$$\dot\bigcup_{i\in A}L_i=\dot\bigcup_{i\in A}R_i=\cR $$ such that for all $i\in A$ and all $x\in \Ga$ we have
\begin{equation}\label{gen_cxbal}
  \sum_{\nu\to\nu\p\in L_i} \pi(x+\nu-\nu\p)\la_{\nu \to \nu\p}(x+\nu -\nu\p)=\pi(x)\sum_{\nu\to\nu\p R_i}\la_{\nu \to \nu\p}(x).
\end{equation}
\end{definition}
\begin{remark} \label{rmk_gen}The notion of generalized balanced includes the following:
\begin{enumerate}
\item reaction balanced with index given by reactions, i.e. the tuples of subsets are $\{(\nu\to\nu\p,\nu\p\to\nu)_{\nu\to\nu\p\in\cR}\}$
\item complex balanced with index given by complexes, i.e. the tuples of subsets are defined for $C\in\cC$\hfill\break $L_C=\{\nu\to\nu\p\in\cR|\nu=C\},R_C=\{\nu\to\nu\p\in\cR|\nu\p=C\}.$
\item reaction vector balanced with index given by $a\in \Z^n$, i.e. the tuples of subsets are defined for $a\in \Z^n$\hfill\break $L_a=\{\nu\to\nu\p\in\cR|\nu-\nu\p=a\},R_a=\{\nu\to\nu\p\in\cR|\nu-\nu\p=-a\}.$
\end{enumerate}
but also combinations and other possibilities.
\end{remark}
In this paper, the following will be often used.
\begin{remark}\label{fact_union}
Let the reactions of a CRN be divided into \color{black}$\cR=\cR_1\cup \cR_2$\color{black}, then it might happen that the stationary \color{black}distribution \color{black}factorises through these reactions. More formally this corresponds to generalised balance with \color{black}$\{(\cR_i,\cR_i)_{i\in\{1,2\}}\}$\color{black}.
\end{remark}
Furthermore generalized balanced distributions on irreducible components give stationary distributions for the reaction network.
\begin{proposition}\cite{hoesslysta}\label{tauto_lemma}
If $(\cG,\k)$ is a CRN with stochastic dynamics on $\Ga$ that is generalized balanced for $\pi$, then $\pi$ is a stationary distribution for $(\cG,\k)$ on $\Ga$.
\end{proposition}

  \subsection{Unions of reaction networks }\label{chpt_union}
   Here we look at the operation of combining two reaction networks. Such operations were already introduced and studied in the deterministic setting in \cite{Heather} where they studied the effects of combining reaction networks in the ODE setting with respect to identifiability, steady-state invariants, and multistationarity. While we will use the same framework, we study stationary properties of the stochastic model under combination and focus only on the two cases of reaction-disjoint and non-reaction-disjoint union.

We next introduce the formalisation of unifying reaction networks.

\begin{definition} \label{def_union}
The {\em union} of reaction networks $\cG_1=(\cS_1, \cC_1, \cR_1)$ and $\cG_2=(\cS_2, \cC_2, \cR_2)$ is 
	\begin{align*}
	\cG_1 \cup \cG_2 ~:=~ \left( \cS_1 \cup \cS_2, ~ \cC_1 \cup \cC_2,~ \cR_1 \cup \cR_2 \right)~.
	\end{align*}
The union $\cG_1 \cup \cG_2$ can be built under different assumptions between the underlying reaction networks $\cG_1, \cG_2$. The following implications holds \cite{Heather}:
	\begin{align*}
	\cS_1 \cap \cS_2 = \emptyset \quad \Rightarrow \quad 
	\cC_1 \cap \cC_2= \emptyset {\rm~~or~~}  \cC_1 \cap \cC_2= \{0\} 
		 \quad \Rightarrow \quad
	 \cR_1 \cap \cR_2= \emptyset ~.
	\end{align*}	
	Consider now taking the union of CRNs with rates $(\cG_1\k_1),(\cG_1\k_1)$, i.e. with $\k_1\in\R_{>0}^{\cR_1},\k_2\in\R_{>0}^{\cR_2}$. We focus on the following two cases.
	\begin{enumerate}
	\item {\em Gluing reaction-disjoint networks:} If the two networks have no reactions in common (i.e., $\cR_1 \cap \cR_2= \emptyset$), then the vector of reaction rates of the union of the reaction networks is equal to $(\k_1,\k_2)\in\R_{>0}^{\cR_1\cup\cR_2}$.
	\item {\em Gluing over reactions:} If the two networks have at least one reaction in common (i.e., $\cR_1 \cap \cR_2 \neq \emptyset$), then the rates of the reactions of the union of the networks which are common reactions (i.e. in $\cR_1\cap\cR_2$) are the sum, i.e.,
 if
 $\nu\to\nu\p\in \cR_1$ with reaction rate $\k^1_{\nu\to\nu\p}$ and $\nu\to\nu\p\in \cR_2$ with reaction rate $\k^2_{\nu\to\nu\p}$, then the reaction rate of $\nu\to\nu\p\in \cR=\cR_1\cup \cR_2$
 is \\
 $\k_{\nu\to\nu\p}=\k^1_{\nu\to\nu\p}+\k^2_{\nu\to\nu\p}$.
	\end{enumerate}
\end{definition}

If the two species sets are disjoint ($\cS_1 \cap \cS_2 = \emptyset$), then the dynamics of the reaction networks $\cG_1$ and $\cG_2$ are independent of each other, hence some properties are directly determined by the dynamics on $\cG_1$ and $\cG_2$ (cf. Remark \ref{simplest_case} for more on this in the stochastic case).  

\begin{remark}\label{rem_union_not_closed} It is easy to see that both detailed balanced and complex balanced reaction networks are not closed under reaction-disjoint unions.
Consider, e.g., the following example:
$$2A \rightleftharpoons A+B,\quad A+2B \rightleftharpoons 3B,$$
with $\cG_1$ the part with two-molecular reactions and $\cG_2$ as the three-molecular reactions. The deficiency of $\cG=\cG_1\cup \cG_2$ is equal to one, hence for almost all parameters it will not be complex balanced. However, it is easy to check that both $\cG_1$ and $\cG_2$ are detailed balanced and hence complex balanced by themselves.
\end{remark}

 \section{Stochastic reaction networks under joins}\label{stochCRN}
  \subsection*{Notation}

 Let $\cG=\cG_1 \cup \cG_2$ be a reaction network obtained from a union of networks as in \color{black}Definition \color{black}\ref{def_union}.
 We denote the projections by
 \begin{center}\hfill\break
  $p_1:\Z^{\cS_1\cup \cS_2}\to \Z^{\cS_1}$\break
 $ p_2:\Z^{\cS_1\cup \cS_2}\to \Z^{\cS_2}
  $\break
   $ p_{12}(=p_{21}):\Z^{\cS_1\cup \cS_2}\to \Z^{\cS_1\cap \cS_2}
  $\break
     $ p_{11}:\Z^{\cS_1\cup \cS_2}\to \Z^{\cS_1\setminus (\cS_1\cap \cS_2)}
  $\break
$ p_{22}:\Z^{\cS_1\cup \cS_2}\to \Z^{\cS_2\setminus (\cS_1\cap \cS_2)}
  $\break
  $p_{S_i}:\Z^{\cS_1\cup \cS_2}\to \Z$   

  \end{center}
  where $p_{S_i}$ is the projection to the $i^{th}$ component.
\subsection{Properties of stochastic dynamics under joins I}\label{stochCRN_join}
We first go through the case of a join where $\cS_1 \cap \cS_2 = \emptyset$ for the sake of exposition and to introduce the reader to the setting. For notations on CTMCs in the context of CRNs we refer to $\S$ \ref{stoch_model}, or, e.g., \cite{norris}.
   \begin{remark}\label{simplest_case} If $\cG=\cG_1 \cup \cG_2$ is such that $\cS_1 \cap \cS_2 = \emptyset$ and $x,y\in \Z_{\geq 0}^\cS$, then 
 $x\to y$ with dynamics of $\cG$ if and only if both $p_i(x) \to_{\cG_i} p_i(y)$ with dynamics of $\cG_i,i=1,2$. 
The decomposition of state space with respect to irreducible components is simple. If $\Ga$ is an irreducible component of $\cG$, then $p_1(\Ga)$ and $p_2(\Ga)$ are irreducible components of $\cG_1$, $\cG_2$ such that $\Ga=p_1(\Ga)\times p_2(\Ga)$. So, for $\Ga$ a positive recurrent irreducible component we have
$$
\pi(x)=\pi_1(p_1(x))\pi_2(p_2(x)),
$$
where $\pi_1,\pi_2$ are the stationary distributions on $p_1(\Ga)$ and $p_2(\Ga)$ of $\cG_1$, $\cG_2$ (there is no normalizing factor since the CTMC is a product). It is easy to see that the stationary distribution on the irreducible component $\Ga$ is generalized balanced with \color{black}$\{(\cR_i,\cR_i)_{i\in \{1,2\}}\}$\color{black} (cf. Remark \ref{fact_union} and Theorem \ref{patch_CRN1} for a proof of a generalisation).
   \end{remark}
   
   \begin{remark}
   Even in the simplest setting of Remark \ref{simplest_case} we can not say much concerning class structure of an $x\in \Z_{\geq 0}^\cS$ given only information about the classes of $p_1(x)$ for $\cG_1$ and $p_2(x)$ for $\cG_2$ (cf., e.g., the simple symmetric random walk on $\Z^d$). In general, $x$ is surely transient for $\cG$ if $p_1(x)$ is transient for $\cG_1$ or $p_2(x)$ is transient for $\cG_2$.
   \end{remark}
  We next establish some simple correspondences for the decomposition of the state space where we omit the proofs.
     \begin{lemma}\label{ess_lemma} The following are equivalent for a CRN $\cG$:
 \begin{enumerate}
 \item \label{Prop0}  $\cG$ is essential.
   \item \label{Prop}For all $x\in \Z_{\geq 0}^\cS$ either there are no active reactions on $x$ or we have that $x\to_{\cG}x\p$ implies $x\p\to_{\cG}x$.
    \item \label{Prop3}For all $\nu\to \nu\p\in \cR$ we have $\nu\p\to_{\cG}\nu$ (i.e. $\nu$ is accessible from $\nu\p$ in $\cG$).
   \end{enumerate}
   \end{lemma}
  
  \begin{lemma}\label{Lemma_ess}
  Consider $\cG=\cG_1 \cup \cG_2$ as in \color{black}Definition \color{black}\ref{def_union} and let $x\in \Ga$ be an element of an irreducible component $\Ga$ of $\cG$.
 \begin{enumerate}[label=(C\arabic{*})]
  \item \label{C1} If $\cG$ is a join of reaction-disjoint networks(cf. Def.  \ref{def_union}), then the following holds:\\
  A reaction $y\to y\p\in\cR_1$ is active on $x$ if and only if it is active on $p_1(x)$.
  \item \label{C2} If both $\cG_1$ and $\cG_2$ are essential, then their union $\cG$ is essential.
   \end{enumerate}
   \end{lemma}
      \begin{remark} If $\cG_1$ is almost essential and $\cG_2$ is essential, their union $\cG$ is not necessarily almost essential. As an example consider the following:
   $$\cG_1 =\{X\to Y, \quad 3Y\to3X\}, \quad \cG_2=\{\emptyset \rightleftharpoons W\}$$
   Since for $\cG$ the following part of state space $\{z\in \Z^3| z_W\geq 0,z_X=0,z_Y=2\}$ is not part of an irreducible component, $\cG$ is not almost essential. In particular \ref{C2} does not extend to almost essential. 
   \end{remark}
   \begin{remark}\label{rmk_ess} Even if $\cG$ is essential, there might be no reaction-disjoint (or non-reaction disjoint) decomposition such that $\cG=\cG_1\cup \cG_2$ with $\cG_1,\cG_2$ essential. As an example, consider, e.g., the following CRN
   $$\emptyset \rightleftharpoons S_1,\quad S_2\to \emptyset,\quad S_1\to S_1+S_2$$
   which can be seen as a simple model for gene-expression \cite{Thattai8614}. In this example the only essential subnetworks are $\emptyset \rightleftharpoons S_1$ and the CRN itself.\color{black}
   \end{remark}
   \begin{lemma}\label{weakly_ess} We have the following implication for a reaction network $\cG$:\\
   $\cG$ reversible $\implies$ $\cG$ weakly reversible $\implies$ \ref{Prop3}. of \color{black}Lemma \color{black}\ref{ess_lemma} holds for $\cG$.\\
   In particular, reversible and weakly reversible reaction networks are essential.
   \end{lemma}
   Also compare \color{black}Lemma \color{black}\ref{weakly_ess} to \cite{disc_Crac}, which contains a similar result (written with different notions).
   Furthermore we need the following Lemma which follows by the definition of irreducible component.
   \begin{lemma}\label{support_lem} Let $\cG=\cG_1 \cup \cG_2$ be as in \color{black}Definition \color{black}\ref{def_union} and consider an irreducible component $ \Ga$ of $\cG$ such that $p_1(\Ga)$ is a union of irreducible components of $\cG_1$ (i.e.  $p_1(\Ga)=\dot\bigcup_{i\in \mathcal{I}}\Ga_i^1)$. Then, if for $x\in \Ga,x\p\in\Z_{\geq 0}^\cS$ we have $p_1(x\p)\to_{\cG_1} p_1(x)$ but $p_1(x)\not\to_{\cG_1} p_1(x\p)$, then $x\p\not\in \Ga$.
   \end{lemma}

  \subsection{Stationary distributions of joins of reaction networks}\label{statCRN}
  Here we will generalise the setting of Remark \ref{simplest_case} in a direction where we can still deduce the form of a stationary distribution of the joined network $\cG=\cG_1\cup \cG_2$ from the combinations of the stationary distributions of the separate reaction networks $\cG_1,\cG_2$. Notice that there are no conditions on the type of kinetics.

  \begin{theorem}\label{patch_CRN1}
  Let $\cG=\cG_1 \cup \cG_2$ be a reaction network obtained from a union of reaction-disjoint networks as in \color{black}Definition \color{black}\ref{def_union} with $\cS_1\cap \cS_2\neq \emptyset$. Let $\Ga$ be an irreducible component of $\cG$. Consider the following assumptions:
  \begin{enumerate}[label=(B\arabic{*})]
  \item\label{B1} Assume 
  $p_1(\Ga)$ 
   is a union of irreducible components of the stochastic dynamics of $\cG_1$ (i.e.  $p_1(\Ga)=\dot\bigcup_{i\in \mathcal{I}}\Ga_i^1)$ with stationary distributions on the irreducible components of the following form
  $$\pi_1(p_1(x))=\frac{1}{Z}\prod_{S_i\in\cS_1\setminus\cS_1\cap\cS_2}f_i(x_i)\prod_{S_i\in\cS_1\cap\cS_2}f^1_i(x_i).
  $$
    \item\label{B2} Assume the same (i.e. as in \ref{B1}) for $\cG_2$, where we denote the stationary distribution on an irreducible component of $\cG_2$ by
    $$\pi_2(p_2(x))=\frac{1}{Z}\prod_{S_i\in\cS_2\setminus\cS_1\cap\cS_2}f_i(x_i)\prod_{S_i\in\cS_1\cap\cS_2}f^2_i(x_i).
  $$
  \item\label{B3} Assume there is an $\a>0$ such that for all $x\in \Ga$ and all $S_i\in\cS_1\cap\cS_2$ we have 
  $$\a f_{i}^1(x_i)=f^2_{i}(x_i).$$
  \end{enumerate}
  If \ref{B1}, \ref{B2} and \ref{B3} are satisfied, then $\cG=\cG_1 \cup \cG_2$ has a product-form stationary distribution of the form
\begin{equation}\label{eq_sol1}
\pi(x)=\frac{1}{Z}\prod_{S_i\in\cS}f_i(x_i),
\end{equation}
where for $S_i\in\cS_1\cap\cS_2$ we set $f_i:=f^2_{i}$ on the irreducible component $\Ga$ if \eqref{eq_sol1} is summable. Furthermore $\cG$ is then generalized balanced for $\pi$ on $\Ga$ with $\{(\cR_i,\cR_i)_{i\in \{1,2\}}\}$.
  \end{theorem}

  \begin{proof}
  It suffices by \color{black}Definition \color{black}\ref{tauto_def} and Proposition \ref{tauto_lemma} to show that
for any $x\in \Ga$ the master equation 

\begin{equation}\label{eq_bal}
  \sum_{\nu\to\nu\p\in \cR_i} \pi(x+\nu-\nu\p)\la_{\nu \to \nu\p}(x+\nu -\nu\p)=\pi(x)\sum_{\nu\to\nu\p\in \cR_i}\la_{\nu \to \nu\p}(x)
\end{equation}
is satisfied with solution \eqref{eq_sol1} for $i\in\{1,2\}$, which corresponds to generalized balanced with $\{(\cR_i,\cR_i)_{i\in \{1,2\}}\}$. 
Note that it is enough to prove it for $\cR_1$. Then we are done by the symmetry of the assumption, and \eqref{eq_sol1} is a stationary distribution, given it is summable.

We next prove that the master equation \eqref{eq_bal} holds true for reaction set $\cR_i=\cR_1$ with solution \eqref{eq_sol1}.
 For $x\in\Ga$ by assumption $p_1(x)\in \Z_{\geq 0}^{\cS_1}$ is in an irreducible component of $\cG_1$. 
If this irreducible component is a singleton set, then the equation is trivially true. There are no active reactions of $\cR_1$ on $x$ and the right side of \eqref{eq_bal} is zero. The left side of  \eqref{eq_bal} is zero as well since these states are transient, i.e. the stationary distribution has no support (cf. \color{black}Lemma \color{black}\ref{support_lem}).
Hence assume it is a non-trivial irreducible component of $\cG_1$,  then inserting the proposed Ansatz \eqref{eq_bal} (modulo normalization) gives
\begin{equation}\label{ME_1}
\begin{aligned} \sum_{\nu\to\nu\p\in \cR_1}\prod_{S_i\in\cS}f_i(x_i+\nu_i-\nu\p_i)\la_{\nu \to \nu\p}(x+\nu -\nu\p)=\\
\prod_{S_i\in\cS}f_i(x_i)\sum_{\nu\to\nu\p\in \cR_1}\la_{\nu \to \nu\p}(x).
\end{aligned}
\end{equation}
Since the reactions in $\cR_1$ do not change the coordinates of $\cS_2\setminus \cS_1$, we have for all $S_i\in\cS_2\setminus \cS_1$ and all $\nu\to\nu\p\in\cR_1$ that $f_i(x_i)=f_i(x_i+\nu_i-\nu\p_i)$, i.e. we can factor the equation as
\begin{equation}\label{ME_2}
\begin{aligned} \prod_{S_i\in\cS_2\setminus \cS_1}f_i(x_i)\sum_{\nu\to\nu\p\in \cR_1}\prod_{S_i\in\cS_1}f_i(x_i+\nu_i-\nu\p_i)\la_{\nu \to \nu\p}(x+\nu -\nu\p)=\\
\prod_{S_i\in\cS}f_i(x_i)\sum_{\nu\to\nu\p\in \cR_1}\la_{\nu \to \nu\p}(x).
\end{aligned}
\end{equation}
By assumption $\prod_{S_i\in\cS_2\setminus \cS_1}f_i(x_i)$ is nonzero (i.e. by contradiction with the assumption on the stationary distribution), 
so \eqref{ME_1} is satisfied if the following holds:
$$\sum_{\nu\to\nu\p\in \cR_1}\prod_{S_i\in\cS_1}f_i(x_i+\nu_i-\nu\p_i)\la_{\nu \to \nu\p}(x+\nu -\nu\p)=$$
$$\prod_{S_i\in\cS_1}f_i(x_i)\sum_{\nu\to\nu\p\in \cR_1}\la_{\nu \to \nu\p}(x).$$
Now we identify the left and the right hand sides of the above equation with the corresponding sides of the master equation from $\cG_1$ with the projection $p_1(x)$ on the irreducible component. Since the transition rates of the reactions of $\cR_1$ only depend on the coordinates of $\cS_1$, they are the same as the transition rates of the master equation from $\cG_1$ under $p_1(x)$ and we get an equality by assumption \ref{B1}.
\end{proof}

\begin{remark}\label{obs_}[Assumptions I] Observe the following.
\begin{itemize}
\item Theorem \ref{patch_CRN1} assumes that the stationary distributions of $\cG_1,\cG_2$ are of product-form. While this is a restriction, current results on form of stationary distributions are mostly in product-form (cf. \cite{anderson2,hoesslysta}). { \color{black}Nonetheless, some examples with stationary distribution of non-product form are available \cite[$\S$ 4.1]{irrev} or \cite{bibb_stat}, but calculating it or even writing it down in small examples is demanding.} 
\item By definition, $p_{12}(\Ga)=p_{21}(\Ga)$, and condition \ref{B3} requires the functions $f^1_{i},f^2_{i}$ with $S_i\in\cS_1\cap\cS_2$ to be proportional on $p_{S_i}(\Ga)\subseteq \Z_{\geq 0}$.
\item Notice that Theorem \ref{patch_CRN1} assumes that the union comes from reaction-disjoint networks as in \color{black}Definition \color{black}\ref{def_union}. \color{black} By the proof of Theorem \ref{patch_CRN1} it would also hold for unions of reaction networks where we glue over reactions with \color{black}Definition \ref{def_union} (and similarly for its consequences, i.e., Theorems \ref{patch_CRN2}, \ref{patch_CRN3}, and Corollary \ref{mult_appl}). However, results on gluing over reactions are only a side product of the intended scope and does not seem very practical at the moment. We refer to Remark \ref{complexity_glue_react} for issues on applicability with respect to decomposing CRNs under gluing over reactions.\color{black}
\end{itemize}
\end{remark}
{\color{black}
\begin{remark}\label{obs_2}[Assumptions II]
 Note that assumption \ref{B3} can be stated more general and Theorem \ref{patch_CRN1} still holds with the same proof, i.e. in the following way:\\
 Assume there are constants $\a_i$ for all $S_i\in\cS_1\cap\cS_2$ with $\a_i>0$ such that for all $x\in \Ga$ we have 
  $$\a_i f_{i}^1(x_i)=f^2_{i}(x_i).$$
 If this more general condition together with \ref{B1}, \ref{B2} still holds, the conclusion of Theorem \ref{patch_CRN1} is maintained with $f_i:=f_i^2$ for $S_i\in\cS_1\cap \cS_2$. Furthermore it is easy to see that this does not influence the summability of \eqref{eq_sol}. The same extension then follows for Theorem \ref{patch_CRN3}.
\end{remark}
}

\color{black}
Furthermore, by the same proof as for Theorem \ref{patch_CRN1} we can conclude the following for a slightly generalised setting (where $f_{11}:\Z^{\cS_1\setminus\cS_2}_{\geq 0}\to \R_{>0}$, etc.).
  \begin{theorem}\label{patch_CRN2}
  Let $\cG=\cG_1 \cup \cG_2$ be a reaction network obtained from a union of reaction-disjoint networks as in Definition \ref{def_union} with $\cS_1\cap \cS_2\neq \emptyset$. Let $\Ga$ be an irreducible component of $\cG$. Consider the following assumptions:
  \begin{enumerate}[label=(B$\p$\arabic{*})]
  \item\label{B1p} Assume 
  $p_1(\Ga)$ 
   is a union of irreducible components of the stochastic dynamics of $\cG_1$ (i.e.  $p_1(\Ga)=\dot\bigcup_{i\in \mathcal{I}}\Ga_i^1)$ with stationary distributions on the irreducible components of the following form
  $$\pi_1(p_1(x))=\frac{1}{Z}f_{11}(p_{11}(x))f_{12}(p_{12}(x)).
  $$
    \item\label{B2p} Assume the same (i.e. as in \ref{B1p}) for $\cG_2$, where we denote the stationary distribution on an irreducible component of $\cG_2$ by
    $$\pi_2(p_2(x))=\frac{1}{Z}f_{21}(p_{21}(x))f_{22}(p_{22}(x)).
  $$
  \item\label{B3p} Assume there is an $\a>0$ such that for all $x\in \Ga$ we have 
  $$\a f_{12}(p_{12}(x))=f_{21}(p_{21}(x)).$$
  \end{enumerate}
  If \ref{B1p}, \ref{B2p} and \ref{B3p} are satisfied, then $\cG=\cG_1 \cup \cG_2$ has a stationary distribution of the form
\begin{equation}\label{eq_sol}
\pi(x)=\frac{1}{Z}f_{11}(p_{11}(x))f_{21}(p_{21}(x))f_{22}(p_{22}(x))
\end{equation}
  on the irreducible component $\Ga$ if \eqref{eq_sol} is summable. 
  \end{theorem}

\color{black}

\subsection{Properties of stochastic reaction networks under joins II}

We want to find sufficient conditions such that the projection
$p_1(\Ga)$ 
   is a union of irreducible components of the stochastic dynamics of $\cG_1$, which is a part of the assumption \ref{B1} of Theorem \ref{patch_CRN1}.
   \begin{lemma}\label{triv_lemma} Let $\Ga$ be an irreducible component of $\cG=\cG_1 \cup \cG_2$. If 
$\cG_1$ is essential,
   then $p_1(\Ga)$ is a union of irreducible components of the stochastic dynamics of $\cG_1$. Note that this holds in particular if $\cG_1$ is weakly reversible (cf. \color{black}Lemma \color{black}\ref{weakly_ess}).
   \end{lemma}
   \begin{proof} Let $x\in \Ga$, and let $p_1(x)$ be the corresponding projected element. We have to show it is part of an irreducible component of $\cG_1$. We distinguish the following two cases:
   \begin{itemize}
   \item If there are no active reactions on $p_1(x)$, then by \color{black}Lemma \color{black}\ref{ess_lemma} $p_1(x)$ is not accessible from any other $z\in \Z_{\geq 0}^{\cS_1}$, hence $p_1(x)$ is an irreducible component.
   \item Assume there are active reactions on $p_1(x)$. Then any other $z\in \Z_{\geq 0}^{\cS_1}$ is accessible from $p_1(x)$ if and only if $p_1(x)$ is accessible from this $z$ by \color{black}Lemma \color{black}\ref{ess_lemma}. Therefore the communicating class of $p_1(x)$ is closed.
   \end{itemize}
   \end{proof}

   Next we further investigate the conditions of the results of $\S$ \ref{statCRN} by focussing in particular on essential reaction networks $\cG$.
   
   \begin{proposition}\label{prop_ess}
   Let $\cG=\cG_1\cup \cG_2$ be an essential reaction network. Then the following conditions are equivalent
   \begin{enumerate}[label=(D\arabic{*})]
   \item\label{ess_1} For every irreducible component $\Ga$ of $\cG$, the projection $p_1(\Ga)$ is a union of irreducible components of $\cG_1$ (i.e.  $p_1(\Ga)=\dot\bigcup_{i\in \mathcal{I}}\Ga_i^1)$.
   \item \label{ess_2} $\cG_1$ is essential.
   \end{enumerate}
   \end{proposition} 
   \begin{proof} For \ref{ess_1} $\implies $ \ref{ess_2} it suffices to observe that the projection $p_1$ is surjective, hence as $\Z_{\geq 0}^\cS$ is a union of irreducible components of $\cG$, we have that $p_1(\Z_{\geq 0}^\cS)=\Z_{\geq 0}^{\cS_1}$ is a union of irreducible components of $\cG_1$. In particular, $\cG_1$ is essential.\hfill\break
   \ref{ess_2} $\implies $ \ref{ess_1} follows from \color{black}Lemma \color{black}\ref{triv_lemma}.
   \end{proof}
      \begin{remark}  Note that this implies in particular that an essential CRN $\cG$ has a decomposition into $\cG=\cG_1\cup \cG_2$ with state space decomposition as in Theorem \ref{patch_CRN1} for every irreducible component if and only if there is a decomposition with both $\cG_1,\cG_2$ essential. Furthermore even if $\cG=\cG_1\cup \cG_2$ and $\cG_1$ are essential, there might still be no such decomposition(cf. the example of Remark \ref{rmk_ess}).
   \end{remark}
Hence, if $\cG=\cG_1 \cup \cG_2$ can be decomposed such that $\cG_1$ and $\cG_2$ are essential, we know by \color{black}Lemma \color{black}\ref{Lemma_ess} that $\cG$ is essential. Furthermore by \color{black}Lemma \color{black}\ref{triv_lemma} the projections of irreducible components of $\cG$ are decomposed into unions of irreducible components of \color{black}$\cG_i,i\in\{1,2\}$\color{black}. Therefore, in this case, we can restate Theorem \ref{patch_CRN1} in a simplified form.

 \begin{theorem}\label{patch_CRN3}
  Let $\cG=\cG_1 \cup \cG_2$ be a reaction network that can be decomposed as a reaction-disjoint union such that $\cG_1,\cG_2$ are essential. Let $\Ga$ be an irreducible component of $\cG$. Assume that 
the irreducible components of
  $p_1(\Ga)$ 
   of $\cG_1$ (i.e.  $p_1(\Ga)=\dot\bigcup_{i\in \mathcal{I}}\Ga_i^1)$ have stationary distributions of the form
  $$\pi_1(p_1(x))=\frac{1}{Z}\prod_{S_i\in\cS_1\setminus\cS_1\cap\cS_2}f_i(x_i)\prod_{S_i\in\cS_1\cap\cS_2}f^1_i(x_i),
  $$
   and the irreducible components of
  $p_2(\Ga)$ 
   of $\cG_2$ (i.e.  $p_2(\Ga)=\dot\bigcup_{i\in \mathcal{I}}\Ga_i^2)$ have stationary distributions of the form
    $$\pi_2(p_2(x))=\frac{1}{Z}\prod_{S_i\in\cS_2\setminus\cS_1\cap\cS_2}f_i(x_i)\prod_{S_i\in\cS_1\cap\cS_2}f^2_i(x_i).
  $$
Furthermore, assume that there is an $\a>0$ such that for all $x\in \Ga$ and all $S_i\in\cS_1\cap\cS_2$ we have 
  $$\a f_{i}^1(x_i)=f^2_{i}(x_i).$$
  
Then, $\cG=\cG_1 \cup \cG_2$ has a product-form stationary distribution of the form
\begin{equation}\label{eq_sol2}
\pi(x)=\frac{1}{Z}\prod_{S_i\in\cS}f_i(x_i),
\end{equation}
where for $S_i\in\cS_1\cap\cS_2$ we set $f_i:=f^2_{i}$ on the irreducible component $\Ga$ 
if \eqref{eq_sol2} is summable.
  \end{theorem}

\color{black}
Then consecutive applications of Theorem \ref{patch_CRN3} along decompositions of essential CRNs gives the following.
\begin{corollary}\label{mult_appl}
Let $\cG$ be a reaction network that can be decomposed as a reaction-disjoint union such that $\cG=\cG_1\cup\cdots\cup \cG_s$ with all the $\cG_j$ essential. Denote by $\cS_j^{only}$ the species that are only in $\cS_j$ and no other $\cS_i, i\neq j$, and by $\cS_j^{shared}$ the species in $\cS_j$ that are also in at least one other $\cS_i,i\neq j$.
Assume that $\Ga$ is an irreducible component of $\cG$ and each $\cG_j$ has product-form stationary distribution of the form 
$$\pi_j(p_j(x))=\frac{1}{Z}\prod_{S_i\in\cS_j^{only}}f_i(x_i)\prod_{S_i\in\cS_j^{shared}}f^j_i(x_i),
  $$
  on its irreducible component in $p_j(\Ga)$ such that, if $\cS_j\cap \cS_k\neq \emptyset$, then there is an $\a>0$ such that for all $S_i \in \cS_j\cap \cS_k$ and all $x\in \Ga$ we have
   $$\a f_{i}^j(x_i)=f^k_{i}(x_i).$$
   Then, $\cG=\cG_1\cup\cdots\cup \cG_s$ has a product-form stationary distribution of the form
\begin{equation}\label{eq_solp}
\pi(x)=\frac{1}{Z}\prod_{S_i\in\cS}f_i(x_i),
\end{equation}
where if $S_i$ is in $\cS_j\cap \cS_k$ 
we choose $f_i:=f_i^{j}$ arbitrary, such that the stationary distribution on $\Ga$ is generalized balanced with $\{(\cR_i,\cR_i)_{i\in \{1,\cdots,s\}}\}$, if \eqref{eq_solp} is summable.
\end{corollary}

\color{black}
 \begin{remark} By the completeness of the results for complex balanced CRN (cf. $\S$ \ref{res_stat}) it is clear we can not say more about complex balanced CRN. The same holds for a similar reason for autocatalytic CRNs since we generalise the underlying proof principle of \cite{hoesslysta}, cf. \color{black}Example \color{black}\ref{auto_ex}. However, we offer a framework that can combine autocatalytic, complex balanced or other CRNs, as long as the stationary distributions are of product form and agree on the species in common. Note that it is easy to find small CRNs beyond complex balance with product form stationary distribution, and we cover only some. In particular there are both reversible, weakly reversible or non-weakly reversible CRN with product-form stationary distributions which can be combined in the framework we developed (cf. $\S$ \ref{stochCRNEx}).
  \end{remark}
\begin{remark} As another example consider 
$$S_0\to S_1 \rightleftharpoon S_2 \rightleftharpoon S_3$$
which is not essential, hence there is no reaction-disjoint decomposition into $\cG=\cG_1\cup \cG_2$ such that both are essential. Hence Theorem \ref{patch_CRN1} still applies while 
Theorem \ref{patch_CRN3} does not. \end{remark}

\begin{remark}[Summability]
\color{black}Note that for summability of \eqref{eq_sol1}, \eqref{eq_sol} or \eqref{eq_sol2} it is necessary that the stationary distributions on the projections are summable. In easy cases with infinite state space summability can possibly be checked by the ratio test, see Remark \ref{rem_ex2} for an example.\color{black}
\end{remark}
\color{black}

\subsection{Decomposing CRNs and applications of Theorem \ref{patch_CRN3}}
While general characterisations of existence for decompositions with a view towards Theorems \ref{patch_CRN1}, \ref{patch_CRN2} ,\ref{patch_CRN3} are not our focus, we remark on several issues. 

Note that two reaction-disjoint CRNs which are given by their reactions $\cR_1,\cR_2$ have a union with reaction set $\cR=\cR_1\cup\cR_2$ if and only if $\cR$ can be decomposed as $\cR=\cR_1\cup\cR_2$ (i.e. with the same $\cR_1,\cR_2$) by a reaction disjoint decomposition. 
Hence for a given set of reactions $\cR$, there are $2^{|\cR|}-1$ such bipartitions of the reactions, which grows exponentially with $|\cR|$. Correspondingly, brute-force algorithms can be given, e.g., for CRNs consisting of a complex balanced and an autocatalytic part. Therefore if stationary distributions for more classes of CRNs (even beyond product-form but as in Theorem \ref{patch_CRN2}) are known, this can be incorporated in a similar way for essential CRNs. Similarly decompositions along Proposition \ref{mult_appl} can be checked. Note that as gluing over reactions is more general, some CRNs might be decomposable in that sense into essential CRNs where it is not possible for reaction-disjoint CRNs. Feasible strategies to cope with such situations are considerably more difficult than reaction-disjoint decompositions, but might be developed at a later point (see Remark \ref{complexity_glue_react}). We further note that characterisations of when such decompositions exist are mostly unknown, and even in the essential case we currently only have characterisations for CRNs with stationary distributions given by Poisson product-form functions by \cite{Cappelletti}. 
\begin{remark}\label{complexity_glue_react}
While the above if and only if statement still holds for CRNs where we glue over reactions,
the number of possible decompositions of a CRN $\cG$ (cf. Definition \ref{def_union}) where we allow gluing over reactions is uncountable, which is not very practical. We henceforth mostly focus on decomposing along reaction-disjoint unions.
\end{remark}

 \section{Applications and examples}\label{stochCRNEx}

We will next go through some examples in order to explain and illustrate the use of the theory developed. \color{black}We mostly focus on mass-action kinetics in examples \ref{ex_class2}, \ref{ex_cx_bal}, \ref{more_compl_cx}, \ref{auto_ex} and \ref{other_ex}, and consider Example \ref{ex_class2} with general kinetics in Example \ref{ex_gen_kin}. We conclude that many such  weakly reversible CRNs of arbitrary deficiency have product-form stationary distribution independent of the rate and independent of the kinetics. While we only used the theory for CRNs, it applies to other stochastic networks and CTMCs as well. \color{black} Furthermore recall that irreducible components of conservative CRNs are finite, hence the limiting distribution is the unique stationary distribution (cf. Remark \ref{rmk_pos_cons}).

\subsection{Examples with Mass-action kinetics}
 
 \begin{example}\label{ex_class2} 
 As a first example consider the following CRN which is reversible and of deficiency two for an application of Theorem \ref{patch_CRN3}.
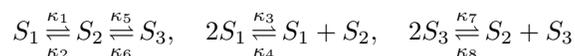
$$
 S_1 \rightleftharpoon^{\k_1}_{\k_2} S_2 \rightleftharpoon^{\k_5}_{\k_6} S_3,\quad 2S_1 \rightleftharpoon^{\k_3}_{\k_4} S_1+S_2,\quad 2S_3 \rightleftharpoon^{\k_7}_{\k_8} S_2+S_3
 $$

 We first decompose $\cG=\cG_1\cup\cG_2$ into two essential CRNs:
$$
 \cG_1:S_1 \rightleftharpoon^{\k_1}_{\k_2} S_2,\quad 2S_1 \rightleftharpoon^{\k_3}_{\k_4} S_1+S_2,\quad \quad \cG_2: S_2 \rightleftharpoon^{\k_5}_{\k_6} S_3,\quad 2S_3 \rightleftharpoon^{\k_7}_{\k_8} S_2+S_3
 $$
 Then we analyse the stationary distributions of $\cG_1,\cG_2$ on their own in order to apply Theorem \ref{patch_CRN3} at the end.
 
 Similar to the example of Remark \ref{rem_union_not_closed}, $\cG_1$ is only for some values detailed balanced. It has a stationary distribution of the form (see Remark \ref{rem_ex2})
 \begin{equation} \label{ME_ex2}\pi(x_1,x_2)=\frac{1}{Z}f_1(x_1)f^2_2(x_2),\end{equation}
on irreducible components $\Ga^1_N=\{x\in\Z_{\geq 0}^2| \sum_{i=1}^2x_i=N \}$,
 where $f_1,f_2^2$ have the following form for $d_1>0$
 $$f_1(x_1)=\frac{d_1^{x_1}}{x_1!}\prod_{l=1}^{x_1} \frac{\k_2+\k_4(l-1)}{\k_1+\k_3(l-1)} ,\quad f^2_2(x_2)=\frac{d_1^{x_2}}{x_2!}.$$
Note that $x_1+x_2$ is constant on the irreducible components $\Ga^1_i$, so also $d_1^{x_1+x_2}$ is a constant along irreducible components.

 Next consider $\cG_2$ 
 %
with stationary distribution  (again with $d_2>0$)
 $$\pi(x_2,x_3)=\frac{1}{Z}f^2_2(x_2)f_3(x_3);$$
 $$ f^2_2(x_2)=\frac{d_2^{x_2}}{x_2!}, \quad f_3(x_3)=\frac{d_2^{x_3}}{x_3!}\prod_{l=1}^{x_3} \frac{\k_6+\k_8(l-1)}{\k_5+\k_7(l-1)}.$$
Now we look at $\cG=\cG_1\cup\cG_2$ in order to apply Theorem \ref{patch_CRN3}. We choose $d_1=d_2=1$ so that the product-functions agree. Then the stationary distribution of $\cG$ is as follows, 
$$\pi(x_1,x_2,x_3)=\frac{1}{Z}f_1(x_1)f^2_2(x_2)f_3(x_3),$$

 where the product form functions are
 $$f_1(x_1)=\frac{1}{x_1!}\prod_{l=1}^{x_1} \frac{\k_2+\k_4(l-1)}{\k_1+\k_3(l-1)} ,\quad f_2(x_2)=\frac{1}{x_2!},\quad f_3(x_3)=\frac{1}{x_3!}\prod_{l=1}^{x_3} \frac{\k_6+\k_8(l-1)}{\k_5+\k_7(l-1).}$$
We further note that the summability in this example is trivial as the irreducible components are finite.
 \end{example}

  \begin{remark}\label{rem_ex2} For $\cG_1$ of \color{black}Example \color{black}\ref{ex_class2} observe the following
 \begin{itemize}
  \item On an irreducible component with a product-form stationary distribution and a conservation relation, we will mostly factor out a constant $d>0$. As an example, consider $\cG_1$ where $x_1+x_2$ is constant on the irreducible components $\Ga^1_i$, so also $d_1^{x_1+x_2}$ is a constant along irreducible components. Then as we divide by the normalising constant the corresponding stationary distributions are all the same for different $d_1>0$. 

 \item $\cG_1$ is reaction vector balanced independently of the rates. We can verify that \eqref{ME_ex2} is reaction vector balance (and hence the stationary distribution) for $\cG$ by checking the following
\begin{align*}
 \pi(x_1+1,x_2-1)(x_1+1)(\k_1+\k_3x_1)= \pi(x_1,x_2)x_2(\k_2+\k_4x_1)\\
 \pi(x_1-1,x_2+1)(x_2+1)(\k_2+\k_4(x_1-1))= \pi(x_1,x_2)x_1(\k_1+\k_3(x_1-1))
\end{align*}

  \item  For $\k_3=\a\k_1,\k_4=\a\k_2$ with $\alpha>0$, $\cG_1$ is detailed (hence complex) balanced, and we can factorize out in $f_1$ from \eqref{ME_ex2} to obtain
 $$f_1(x_1)=\frac{d_1^{x_1}}{x_1!}\big(\frac{\k_1}{\k_2}\big)^{x_1}, \quad f^2_2(x_2)=\frac{d_1^{x_2}}{x_2!}.$$
 To transform this into a standard form, we can choose $d_1=\k_2$.
\item We can join $\cG_1$ with the following essential CRN $\cG_2$
$$\emptyset \rightleftharpoon_{\k_-}^{\k_+} S_2$$
with stationary distribution $\pi(x_2)=\frac{1}{Z}f_2^2(x_2)$, with $f_2^2(x_2)= \frac{c_2^{x_2}}{x_2!}$ and where $c_2=\frac{\k_+}{\k_-}$ is a point of complex balance. Then choosing $d_1=c_2$ makes the product-form functions $f^1_2,f^2_2$ equal. Therefore if the following is summable, it is the stationary distribution
$$\pi(x_1,x_2)=\frac{1}{Z}f_1(x_1)f_2^2(x_2),$$
where we have to check that the following sum is finite $$\sum_{(x_1,x_2)\in\Z_{\geq 0}^2}\frac{c_2^{x_1}}{x_1!}\prod_{l=1}^{x_1} \frac{\k_1+\k_3(l-1)}{\k_2+\k_4(l-1)}\frac{c_2^{x_2}}{x_2!}=\exp(c_2)\sum_{x_1\in\Z_{\geq 0}^1}\frac{c_2^{x_1}}{x_1!}\prod_{l=1}^{x_1} \frac{\k_1+\k_3(l-1)}{\k_2+\k_4(l-1)}.$$
Therefore it is easy to see, e.g. by the ratio test for series, that the series converges for all positive rate parameters.
 \end{itemize}
  \end{remark}

 \begin{example}\label{ex_cx_bal} Consider the CRN of \cite[Example 4.4]{hoesslysta}. 
 $$\begin{tikzcd} [ row sep=1em,
  column sep=1em]
S_1\arrow[r, shift left=1ex,"\a_{1,2}^1"]  &S_2 \arrow[l,"\a_{2,1}^1"]\\ 
 S_1+S_2 \arrow[r, "\a_{1,2}^2"]  &2S_2 \\
 2S_1 \arrow[rr, shift left=1ex,"\k_1"] 
&& 2S_3 \arrow[dl, "\k_2"] \\
& S_1+S_3\arrow[ul,"\k_3"]
\end{tikzcd}
$$
Then $\cG_1$ is autacatalytic and corresponds to reactions between $S_1,S_2$ and $\cG_2$ is complex balanced and corresponds to the reactions between $S_1,S_3$.
Hence $\cG_1,\cG_2$ are essential, and we may apply Theorem \ref{patch_CRN3} after deriving the stationary distributions of $\cG_1$ and $\cG_2$, giving an easy way to compute the stationary distribution of \cite[Example 4.4]{hoesslysta}.

\end{example}

This shows how to systematically decompose some examples of CRNs into smaller parts where the stationary distribution is known and of product-form. As another example consider the following CRN where we glue along two species.
\begin{example}\label{more_compl_cx} Let $\cG_1$ be the following essential CRN
$$S_1\rightleftharpoons S_2 \rightleftharpoon S_3 \rightleftharpoons S_4, 
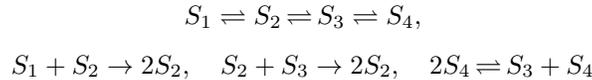
$$
$$S_1+S_2\to 2S_2,\quad S_2+S_3\to 2S_2, \quad 2S_4 \rightleftharpoon S_3+S_4$$
We can choose the parameters to obtain an autocatalytic CRN according to \cite{hoesslysta} on $S_1,S_2,S_3$, and join it with the CRN on $S_3,S_4$ (which was \color{black}Example \color{black} \ref{ex_class2}) with stationary distribution of the form
$$\pi(x_1,x_2,x_3,x_4)=\frac{1}{Z}f^1_1(x_1)f_2(x_2)f^1_3(x_3)f_4(x_4)$$
with $f_1^1(x_1),f^1_3(x_3)$ of Poisson product-form.

Consider as $\cG_2$ the following complex balanced (hence weakly reversible, essential) CRN:
$$3S_3\rightleftharpoon 3S_5
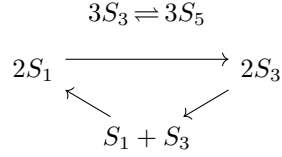
$$
$$\begin{tikzcd} [ row sep=1em,
  column sep=1em]
 2S_1 \arrow[rr, shift left=1ex,] 
&& 2S_3 \arrow[dl, ] \\
& S_1+S_3\arrow[ul,]
\end{tikzcd}
$$
with stationary distributions of the form
$$\pi(x_1,x_3,x_5)=\frac{1}{Z}f^2_1(x_1)f^2_3(x_3)f_5(x_5)$$
$$ f^2_1(x_1)=\frac{(c_1d_2)^{x_1}}{x_1!}, f^2_3(x_3)=\frac{(c_3d_2)^{x_3}}{x_3!}, f_5(x_5)=\frac{(c_5d_2)^{x_5}}{x_5!} $$
with $(c_1,c_3,c_5)$ a point of complex balance. Then, if the rates of $\cG_1,\cG_2$ are such that the product-form functions $f^1_1,f^1_3$ and $f^2_1,f^2_3$ can be chosen to be the same, we can give the stationary distribution of $\cG=\cG_1\cup \cG_2$.
\end{example}

 \begin{example}\label{auto_ex} Now we consider autocatalytic CRNs \cite{hoesslysta}. Interacting particle systems of that form are used in inclusion processes from statistical physics and the modelling of ants and swarms \cite{inclusion_proc,hoesslysta,swarm}. Consider a CRN $\cG_1$ on 2 species CRN with the reactions
 $$S_1 \rightleftharpoon_{\a^1_{1,2}}^{\a^1_{2,1}} S_2$$
 together with reactions in any of the following form
 \begin{equation} \label{basic_re}S_2+(m-1)S_1\too^{\a^m_{2,1}} mS_1, \end{equation}
where $m\geq1$. Note that such CRNs are essential. Then obtaining the stationary distributions for such CRNs on two species and assembling them
leads to the stationary distributions of autocatalytic CRNs from \cite{hoesslysta}, which can also be obtained via the decomposition into joins with Theorem \ref{patch_CRN3}.
 \end{example}


\color{black}

\color{black}  

\begin{example}\label{other_ex} While ergodic conservative CRNs with product-form stationary distributions have a degree of freedom to choose the product-form function, that is not the case for 
other ergodic CRNs. As an example consider
$$
 S_1 \rightleftharpoon^{\k_1}_{\k_2} S_2\rightleftharpoon_{\k_+}^{\k_-}\emptyset,\quad 2S_1 \rightleftharpoon^{\k_3}_{\k_4} S_1+S_2,\quad
 S_1 \rightleftharpoon^{\k_6}_{\k_7} S_3\rightleftharpoon_{\k_{+,2}}^{\k_{-,2}}\emptyset,\quad 2S_1 \rightleftharpoon^{\k_8}_{\k_9} S_1+S_3,
 $$
 which can be decomposed into two CRNs along Example \ref{ex_class2} and Remark \ref{rem_ex2}. Then, application of Theorem \ref{patch_CRN1} requires that the parameters match in some sense and further summability also has to be taken care of.
\end{example}

\subsection{Examples with more general kinetics}
We recall the setting of more general intensity functions from \cite{anderson2} which are given as
\begin{equation}
  \lambda_{\nu\to\nu\p}(x) = \k_{\nu\to\nu\p} \prod_{S_i\in \cS} \prod_{j=0}^{\nu_{i}-1}
  \theta_i(x_i - j) 
  \label{eq:gen_kin}
\end{equation}
where the $\k_{\nu\to\nu\p}$ are positive reaction rates and $\theta_i : \Z \to
\R_{\ge 0}$ are such that $\theta_i(x) = 0$ if and only if $x \le 0$ (we use the convention
that $\prod_{j=0}^{-1}a_j = 1$ for any $\{a_j\}$).
Typical kinetics used in mathematical biology are, e.g.,
$$x\mapsto \frac{x^m}{k^m+x^m},\quad \quad \quad x\mapsto \frac{k_1^m}{k_2^m+x^m},$$
called Hill-type I/II in \cite{grima_kin}, where $m$ is an integer and $k,k_1,k_2$ are positive constants. The first specialises to stochastic Michaelis-Menten kinetics for $m=1$ \cite{anderson2}. 

\begin{example}\label{ex_gen_kin}
Consider again Example \ref{ex_class2}, i.e.,
$$
 S_1 \rightleftharpoon^{\k_1}_{\k_2} S_2,\quad 2S_1 \rightleftharpoon^{\k_3}_{\k_4} S_1+S_2
 $$
  but with general kinetics $\theta_1,\theta_2$. The irreducible components are as in Example \ref{ex_class2}, and the stationary distribution can again be given via the reaction vector balance equations of Remark \ref{rem_ex2}, giving
 \begin{equation} \label{ME_ex2_gen}\pi(x_1,x_2)=\frac{1}{Z}f_1(x_1)f^2_2(x_2),\end{equation}
 where $f_1,f_2^2$ have the following form
 $$f_1(x_1)=\frac{d_1^{x_1}}{\prod_{l=1}^{x_1}\theta_1(l)}\prod_{l=1}^{x_1} \frac{\k_2+\k_4\theta_1(l-1)}{\k_1+\k_3\theta_1(l-1)} ,\quad f^2_2(x_2)=\frac{d_1^{x_2}}{\prod_{l=1}^{x_2}\theta_2(l)}.$$
  
\end{example}

Hence from Example \ref{ex_gen_kin} (also see Example \ref{ex_class2}) it is easy to see that we can assemble arbitrary CRNs of this form with product form stationary distribution independently of the rates via Theorem \ref{patch_CRN3}. This then gives the following.
\begin{corollary}\label{suff1}
Independent of the kinetics (but with $\theta_2$ fixed), any CRN that is a disjoint union of CRNs of the form 
$$
 S_2 \rightleftharpoon S_i,\quad 2S_i \rightleftharpoon S_2+S_i$$
 for $i\neq 2$ has product-form stationary distribution independent of the rates.
\end{corollary}
\begin{remark}\label{comp_cx}[Compatibility with complex balance in $S_2$]\hfill\break
Let $\cG_1$ be a RN obtained from Corollary \ref{suff1} and $\cG_2$ be a weakly reversible, deficiency zero CRN that is conservative with kinetic functions as $\theta_2$ for species $S_2$ such that the only species in common between $\cG_1,\cG_2$ is $S_2$. Then, by \cite[Theorem 6.1]{anderson2}, Corollary \ref{suff1} and Theorem \ref{patch_CRN3} the CRN $\cG_1\cup\cG_2$ has product form stationary distribution 
with the product-form function in $S_2$ given by $f_2(x)=\frac{c_2^x}{\prod_{l=1}^{x}\theta_2(l)}$, where $c\in \R^{\cS_2}_{\geq 0}$ is a point of complex balance for $\cG_2$. 
\end{remark}
\begin{remark}
Note that it is usually not the case that RNs with different kinetics can be joined with matching product-form functions. This comes from the fact that for example if restricting to complex balanced CRNs with Mass-action kinetics and stochastic Michaelis-Menten kinetics, the corresponding product-form functions are different(see \cite[Theorem 6.1]{anderson2}).
\end{remark}

 \bibliographystyle{plain}

 \bibliography{references} 
 
\end{document}